\documentclass{article}
\usepackage{amsmath}
\usepackage{amssymb}
\usepackage{titlesec}
\usepackage{graphicx}
\usepackage{fancyhdr}
\usepackage{latexsym}
\usepackage{mathrsfs}
\usepackage{booktabs}
\usepackage{mathrsfs}
\usepackage{mathtools}
\usepackage{amsthm}
\usepackage{wasysym}
\usepackage{cite}
\usepackage{color}

\usepackage{todonotes}
\usepackage{amsfonts}
\usepackage{amsmath,amssymb}
\usepackage{graphics}
\usepackage{multimedia}
\usepackage{graphics}
\usepackage{cite}
\usepackage{latexsym}
\usepackage{amsmath}
\usepackage{amssymb}
\usepackage[dvips]{epsfig}
\usepackage{amscd}
\numberwithin{equation}{section}

\newtheorem{theorem}{Theorem}[section]
\newtheorem{lemma}{Lemma}[section]
\newtheorem{remark}{Remark}[section]
\newtheorem{proposition}{Proposition}[section]

\newcommand{\dv}{\text{div}}

\title{Global large strong solution of the 3D inhomogeneous Navier-Stokes equations with density-dependent viscosity}
\author{Xiangdi H{\small UANG}$^{c}$, Jiaxu L{\small I}$^{b}$, Rong Z{\small HANG}$^{a,d}$ \thanks{Email addresses: xdhuang@amss.ac.cn (X. D. Huang), Jiaxvlee@gmail.com (J. X. Li), rzhang0921@gmail.com (R. Zhang). }  \\ 
{\normalsize a. School of Mathematics and Computer Sciences,}\\
{\normalsize Nanchang University, Nanchang 330031, P. R. China;}\\
{\normalsize b.  The Institute of Mathematical Sciences,}\\
{\normalsize  The Chinese University of Hong Kong,
Shatin, N.T.;}\\
{\normalsize c. Institute of Mathematics, Academy of Mathematics and Systems Sciences,}\\
{\normalsize Chinese Academy of Sciences, Beijing 100190, China;}\\
{\normalsize d. Institute of Mathematics and Interdisciplinary Sciences,}\\
{\normalsize Nanchang University, Nanchang 330031, P. R. China.}
}
\date{}

\begin{document}
\maketitle
\begin{abstract}
This paper concerns the Dirichlet problem of three-dimensional inhomogeneous Navier-Stokes equations with density-dependent viscosity. When the viscosity coefficient $\mu(\rho)$ is a power function of the density ($\mu(\rho)=\mu\rho^\alpha$ with $\alpha>1$), it is proved that the system will admit a unique global strong solution as long as the initial data are sufficiently large.  This is the first result concerning the existence of large strong solution for the inhomogeneous Navier-Stokes equations in three dimensions.

\end{abstract}

\section{Introduction}
The Navier–Stokes equations are usually used to describe the motion of fluids. In particular, for the study of multiphase fluids without surface tension, the following density-dependent Navier–Stokes equations act as a model on some bounded domain $\Omega\subset\mathbb{R}^3$,
\begin{equation}\label{ins}
	\left\{ \begin{array}{l}
		\rho_t+\mathrm{div}(\rho u)=0,~~\mathrm{in}~\Omega\times [0,T],\\
		(\rho u)_t+\mathrm{div}(\rho u\otimes u)+\nabla P-\mathrm{div}(2\mu(\rho)d)=0,~~\mathrm{in}~\Omega\times [0,T], \\
        \dv u=0,~~\mathrm{in}~\Omega\times [0,T]\\
        \rho(0,x)=\rho_0(x),~~u(0,x)=u_0(x),~~\mathrm{in}~\Omega,
	\end{array} \right.
\end{equation}
where $t\geq0, x=(x_1,x_2,x_3)\in\Omega$ are time and space variables, respectively. $\rho=\rho(x,t)$, $u=(u_1(x,t),u_2(x,t),u_3(x,t))$ and $P$ represent, respectively, the density, the velocity and the pressure of the fluid.
\begin{equation}
	d=\frac12 \left(\nabla u+(\nabla u)^\top\right),
\end{equation}
is the deformation tensor. $\mu(\rho)$ stands for the viscosity and is a function of $\rho$, which is assumed to satisfy
\begin{equation}\label{vis-d}
	\mu(\rho)=\mu\rho^\alpha,\ \mu>0,\ \alpha\geq 0. 
\end{equation}
In this paper, we study the initial boundary value problem to the system \eqref{ins}-\eqref{vis-d} with Dirichlet boundary condition:
\begin{equation}\label{bc}
    u=0,~~\mathrm{on}~\partial\Omega\times [0,T].
\end{equation}

There is a lot of literature on the mathematical study of nonhomogeneous incompressible
flow. In particular, the system \eqref{ins} with constant viscosity has been
investigated extensively. On the one hand, in the absence of a vacuum, Kazhikov proved the global
existence of weak solutions and the local existence of strong ones in \cite{antontsev1973mathematical}. Later,
Ladyzhenskaya–Solonnikov \cite{ladyzhenskaya1978unique} first proved the global
well-posedness of strong solutions to the initial boundary value problems in both
2D bounded domains (for large data) and 3D ones (with initial
velocity small in suitable norms).
Recently, there have been many subsequent works on the global well-posedness results with  small initial data in critical spaces (see
\cite{abidi2011decay,10.57262/ade/1355867948,danchin2012lagrangian,huang2013global} and the references therein).

In general, as long as viscosity $\mu(\rho)$ depends on density $\rho$, things become more complicated due to the strong coupling between the viscosity and the density. Desjardins \cite{desjardins1997regularity} proved the global weak solution for the two-dimensional case provided that the viscosity function $\mu(\rho)$ is a small perturbation of a positive constant in $L^\infty$-norm. Recently, much progress has been made on the well-posedness of strong solutions to \eqref{ins} under some smallness conditions on the initial data (see \cite{abidi2015global,abidi2015global1,cho2004unique,lu2019local,zhang2015global,he2021global} and the references therein).

In this paper, we would like to establish the global existence of strong solutions to the system \eqref{ins} as long as the density is large enough which is a big contrast to the classical result where the initial data is a small perturbation of equilibrium states. Let's first explain our motivation and observation. 

Recall the definition of Reynolds number defined in \cite{sommerfeld1909beitrag}:
\begin{equation}
    Re=\frac{uL}{\nu}=\frac{\rho u L}{\mu},
\end{equation}
where 
\begin{itemize}
    \item $\rho$ is the density of the fluid,
    \item $u$ is the flow speed, 
    \item $L$ is a characteristic length,
    \item $\mu$ is the dynamic viscosity of the fluid,
    \item $\nu$ is the kinematic viscosity of the fluid.
\end{itemize}  
When $\mu(\rho)$ take the form of \eqref{vis-d}, suppose $L$ is a constant in a bounded domain, formally the Reynolds number is reduced to
\begin{equation}
    Re=\frac{L}{\mu}\frac{u}{\rho^{\alpha-1}},
\end{equation}
no matter how fast the flow speed is, $Re$ will be small provided $\alpha>1$ and the density $\rho$ is large enough.
At low Reynolds numbers, the flows behave like a laminar flow which is believed to be stable as many experiments revealed. 

So this gives us a hint that the flow may be global stable as long as the density is large at any time. Fortunately, this condition can be easily guaranteed by the density equation by assuming the initial density is large enough. The purpose of this paper is to verify this intuition. 

From mathematical point of view, \eqref{ins}$_2$ can be rewritten as 
\begin{equation}\label{sto}
	-\mu\Delta u+\nabla\biggl(\frac{P}{\rho^\alpha}\biggr)=\rho^{1-\alpha}\dot{u}+2\mu\alpha \rho^{-1}\nabla \rho \cdot d
    +\alpha \rho^{-1}\nabla \rho \frac{P}{\rho^\alpha},
\end{equation}
where 
\begin{equation}
    \dot u\triangleq u_t+u\cdot\nabla u
\end{equation}
denote the material derivatives. Each term of the right-hand side of Stokes system \eqref{sto} is a small perturbation term for large density and $\alpha>1$.

Indeed, such a result was verified in the compressible framework. For the compressible Navier-Stokes equations with variable viscosity coefficients with $\mu(\rho)=\mu\rho^{\alpha},\lambda(\rho)=\lambda\rho^{\alpha}, P(\rho)=a\rho^{\gamma}$, Yu \cite{yu2023global} proved the global existence of strong solutions when the initial density is large enough as long as 
\begin{equation}\label{yu1}
    \frac43<\gamma\le \alpha\le \frac53, \alpha+4\gamma>7, \alpha+\gamma\le 3,
\end{equation}
or 
\begin{equation}\label{yu2}
    \frac32<\gamma\le \alpha\le 2, \alpha+\frac23 \gamma>3.
\end{equation}

Before stating the main results, we explain the notation and conventions used throughout this paper. 
Denote
\begin{equation*}
	\int f \mathrm{dx}=\int_{\Omega} f\mathrm{dx},
\end{equation*}
For a positive integer $k$ and $p\geq1$, we denote the standard Lesbegue and Sobolev spaces as follows:
\begin{equation*}
\begin{gathered}
\|f\|_{L^p}=\|f\|_{L^p(\Omega)},\ \|f\|_{W^{k,p}}=\|f\|_{W^{k,p}(\Omega)},\ \|f\|_{H^k}=\|f\|_{W^{k,2}(\Omega)},\\
C^\infty_{0,\sigma}=\{f\in C^\infty_0(\Omega):\mathrm{div}f=0\}, \ H_0^1=\overline{C^\infty_0}, \\
H_{0,\sigma}^1=\overline{C^\infty_{0,\sigma}},\ \mathrm{closure}\ \mathrm{in}\ \mathrm{the}\ \mathrm{norm}\ \mathrm{of}\  H^1.    
\end{gathered}
\end{equation*}

 Here is our main result.
\begin{theorem}\label{global} 
    Let $\Omega$ be a bounded smooth domain in $\mathbb{R}^3$. Assume that
    \begin{equation}\label{vis-c}
        \alpha>1.
    \end{equation}
    Given constants
    \begin{equation}
        \bar\rho> 1,\quad C_0\geq 1.
    \end{equation}
    Suppose that the initial data $(\rho_0,u_0)$ satisfies 
    \begin{equation}\label{ia}
		\bar \rho \le \rho_0\le C_0 \bar \rho,\quad \rho_0 \in 
		W^{1,q},\ 3<q<6,\quad  u_0 \in H_{0,\sigma}^1 \cap H^2.
    \end{equation}
    Then there exists a positive constant $\Lambda_0$ depending only
	on $C_0,\mu,\alpha,$ $\|\nabla\rho_0\|_{L^q}, \| u_0\|_{H^2}$ and $\Omega$ such that if
	\begin{equation}\label{ini}
		\bar\rho \ge \Lambda_0(\Omega,C_0,\mu,\alpha,\|\nabla\rho_0\|_{L^q}, \| u_0\|_{H^2}),
	\end{equation}
    then the problem  (\ref{ins})-(\ref{bc}) admits a unique global strong solution $(\rho,u)$ in $\Omega\times(0,\infty)$ satisfying 
    \begin{equation}
    \left\{ \begin{array}{l}
    \rho\in C([0,\infty);W^{1,q}),\
    \nabla u, P\in C([0,\infty);H^1)\cap L^2((0,\infty);W^{1,q}),\\
    \rho_t \in C([0,\infty);L^q),\
    \sqrt{\rho}u_t\in L^\infty((0,\infty);L^2),\ 
    u_t\in L^2((0,\infty);H^1_0).
    \end{array} \right.
\end{equation}
\end{theorem}

\begin{remark}
    Conditions \eqref{vis-c} and \eqref{ini} imply that the initial Reynolds number is suitably small, so the conclusion of Theorem 1.1 is equivalent to proving a fact that the flow is globally stable when the initial Reynolds number is small enough. 
\end{remark}
\begin{remark}
    Theorem \ref{global} can also be applied to boundary condition \eqref{ins}$_4$ when replaced by Navier-slip boundary conditions and periodic ones. However, the Cauchy problem presents essential difficulties that will be addressed in future work.
\end{remark}
We now provide our analysis and commentary on the key aspects of this paper. 
The main idea is to use time-weighted energy estimates to the compressible Navier–Stokes equations established by Hoff \cite{hoff1995global}, which is successfully used to the inhomogeneous incompressible Navier-Stokes equations by Huang-Wang \cite{craig2013global,huang2014global,huang2015global} and many other works, see \cite{hao2024globalwellposednessinhomogeneousnavierstokes,paicu2013global} and references therein.

Let's briefly sketch the proof. First we assume that $\mathcal{E}_\rho(T)$ is less than 3$\mathcal{E}_\rho(0)$  and $\mathcal{E}_u(T)$ is less than 3$\mathcal{E}_u(0)$,  then we prove that in fact $\mathcal{E}_\rho(T)$ is less than 2$\mathcal{E}_\rho(0)$ and $\mathcal{E}_u(T)$ is less than 2$\mathcal{E}_u(0)$ under the assumption that the initial density is large enough. On the other hand, the control of $\|\nabla u\|_{L^1_tL^\infty_x}$ leads to uniform estimates for other higher-order quantities, which guarantees the extension of local strong solutions.
One of the main ingredients is a time-independent estimate which is essential due to exponential time decay estimates for $u$ in a bounded domain. However, this is not the case for the whole space.

\section{Preliminaries}
First, the following local existence theory, where the initial density is strictly away from vacuum, can be shown by similar arguments as in Cho and Kim \cite{cho2004unique}:
\begin{lemma}\label{local}
    Assume that the initial data $(\rho_0, u_0)$ satisfies the regularity condition \eqref{ia}.
    Then there exists a small time $T$ and a unique strong solution $(\rho, u, P)$ to the initial boundary value problem (\ref{ins})--(\ref{bc}) such that
\begin{equation}\label{l-r}
    \left\{ \begin{array}{l}
    \rho\in C([0,T];W^{1,q}),\
    \nabla u, P\in C([0,T];H^1)\cap L^2([0,T];W^{1,q}),\\
    \rho_t \in C([0,T];L^q),\
    \sqrt{\rho}u_t\in L^\infty([0,T];L^2),\ 
    u_t\in L^2([0,T];H^1_0).
    \end{array} \right.
\end{equation}
Furthermore, if $T^\ast$is the maximal existence time of the local strong solution $(\rho, u)$, then either
$T^\ast=\infty$ or
\begin{equation}\label{blow-up}
    \sup_{0\leq t\leq T^\ast}\left(\|\nabla\rho\|_{L^q}+\|\nabla u\|_{L^2}\right)=\infty.
\end{equation}
\end{lemma}
In this paper, we will employ Bovosgii's theory which can be found in \cite{seregin2014lecture}.
\begin{lemma}\label{bovosgii}
    Let $\Omega$ be a bounded domain with Lipschitz boundary, $1 < p < \infty$. Given $b \in L^p(\Omega)$ with $\int_\Omega b dx=0$, there exists $v \in W_0^{1,p}(\Omega)$ with the following properties:
    \begin{equation*}
        \mathrm{div} v=b,
    \end{equation*}
    in $\Omega$, and
    \begin{equation}
        \|\nabla v\|_{L^p(\Omega)}\leq C(p)\|b\|_{L^p(\Omega)}.
    \end{equation}
\end{lemma}
Also, the well-known Gagliardo-Nirenberg inequality \cite{ladyzhenskaia1968linear} will be frequently used in this paper.
\begin{lemma}\label{G-N}
Assume that $\Omega$ is a bounded Lipschitz domain in $\mathbb{R}^3$. Let 1 $\leq q  +\infty$ be a positive extended real quantity. Let j and m  be non-negative integers such that $ j < m$. Furthermore, let $1 \leq r \leq \infty$ be a positive extended real quantity,  $p \geq 1$  be real and  $\theta \in [0,1]$  such that the relations
\begin{equation}
    \dfrac{1}{p} = \dfrac{j}{n} + \theta \left( \dfrac{1}{r} - \dfrac{m}{n} \right) + \dfrac{1-\theta}{q}, \qquad \dfrac jm \leq \theta \leq 1
\end{equation}
hold. Then, 
\begin{equation}
    \|\nabla^j u\|_{L^p(\Omega)} \leq C\|\nabla^m u\|_{L^r(\Omega)}^\theta\|u\|_{L^q(\Omega)}^{1-\theta} + C_1\|u\|_{L^q(\Omega)}
\end{equation}
where $u \in L^q(\Omega)$  such that  $\nabla^m u \in L^r(\Omega)$. Moreover, if $q>1$ and $r>3$,
\begin{equation}
\|u\|_{C(\bar{\Omega})}\leq C\|u\|^{q(r-3)/(3r+q(r-3))}_{L^q}\|\nabla u\|^{3r/(3r+q(r-3))}_{L^r}+C_2\|u\|_{L^q}.
\end{equation}
where $u \in L^q(\Omega)$  such that  $\nabla u \in L^r(\Omega)$.
In any case, the constant $C > 0$  depends on the parameters $j,\,m,\,n,\,q,\,r,\,\theta$, on the domain $\Omega$, but not on $u$.

In addition, if $u\cdot n|_{\partial\Omega} = 0$, we can choose $C_1 = C_2 = 0$.
\end{lemma}

\section{A priori estimates}

For any fixed time $T>0$, $(\rho,u,P)$ is the unique local strong solution to \eqref{ins}-\eqref{bc} on $\Omega\times (0,T]$ with initial data $(\rho_0,u_0)$ satisfying \eqref{ia}, which is guaranteed by lemma 
\ref{local}.

Define
\begin{gather}\label{As1}
	\mathcal{E}_\rho(T) \triangleq \sup_{t\in[0,T] }\|\nabla \rho \|_{L^q},\\
	\mathcal{E}_u(T) \triangleq \bar\rho^{\alpha}\sup_{t\in[0,T] }\|\nabla u \|_{L^2}^2
	+  \int_0^{T}  \|\sqrt{\rho} u_t \|_{L^2}^2dt.
\end{gather}

We have the following key proposition.

\begin{proposition}\label{pr}
Under the conditions of Theorem \ref{global}, there exists a  positive constant $\Lambda_0$ depending on $\Omega,C_0,\mu,\ \alpha$, and $\|\nabla\rho_0\|_{L^q}, \| u_0\|_{H^2}$ such that if $(\rho,u,P)$ is a smooth solution to the problem (\ref{ins})--(\ref{bc}) on $\Omega\times (0,T]$ satisfying
	\begin{equation}\label{a1}
		\mathcal{E}_\rho(T) \le 3\mathcal{E}_\rho(0),~ 
		\mathcal{E}_u(T) \le 3\mathcal{E}_u(0), 
	\end{equation}
	then the following estimates hold:
	\begin{equation}\label{a2}
		\mathcal{E}_\rho(T) \le 2\mathcal{E}_\rho(0),~
		\mathcal{E}_u(T) \le 2\mathcal{E}_u(0),
	\end{equation}
	provided 
	\begin{equation}
		\bar\rho\ge \Lambda_0(\Omega,C_0,\mu,\alpha,\|\nabla\rho_0\|_{L^q}, \| u_0\|_{H^2}).
	\end{equation}     
\end{proposition}
First, as the density satisfies the transport equation \eqref{ins}$_1$ and making use of \eqref{ins}$_3$, one has the following lemma
\begin{lemma}
It holds that 
	\begin{equation}\label{2.3}
		\bar \rho \le \rho \le C_0 \bar \rho,\quad (x,t)\in \Omega \times [0,T].
	\end{equation}
\end{lemma}
Next, the basic energy inequality of the system \eqref{ins} reads
\begin{lemma}
	It holds that
	\begin{equation}\label{basic-est}
		\sup_{0\le t\le T} \bar \rho \| u\|_{L^2}^2   + \bar\rho^{\alpha} \int_{0}^{T} \|\nabla u\|_{L^2}^2 dt \le C\bar\rho.
	\end{equation}
\end{lemma}

\begin{proof}
	Multiplying \eqref{ins}$_2$ by $u$ and integrating the resultant equation, we obtain after integration by parts that 
	\begin{equation}
		\begin{aligned}
			&\frac{1}{2} \frac{d}{dt} \|\sqrt{\rho} u\|_{L^2}^2 + \frac{\mu}{2^\alpha} \bar \rho^{\alpha}\|\nabla u\|_{L^2}^2 \\
			\le &\frac{1}{2} \frac{d}{dt} \|\sqrt{\rho} u\|_{L^2}^2 + \int 
			\mu\rho^\alpha|d|^2  dx=0.
		\end{aligned}
	\end{equation}
	Integrating the above inequality over $(0,T]$ leads to 
	\begin{equation}
		\begin{aligned}
			\sup_{0\le t\le T} \bar \rho \| u\|_{L^2}^2   + \bar\rho^{\alpha} \int_{0}^{T} \|\nabla u\|_{L^2}^2 dt \le C\bar\rho \| u_0\|_{L^2}^2.
		\end{aligned}
	\end{equation}
	The proof is completed.
\end{proof}
High-order a priori estimates rely on the following regularity results for density-dependent Stokes equations.
\begin{lemma}\label{stokes-e}
Assume that $\rho\in W^{1,q}, 3<q<6$, and $\bar \rho \le \rho \le C_0 \bar \rho$.
Let $(u, P) \in H_{0,\sigma}^1\times L^2$ be the unique weak solution to the boundary value problem
\begin{equation}\label{stokes}
	\left\{ \begin{array}{l}
		-\mathrm{div}(2\mu\rho^\alpha d)+\nabla P=F,~~\mathrm{in}~\Omega, \\
		\mathrm{div} u=0,~~\mathrm{in}~\Omega,\\
        \int \frac{P}{\rho^\alpha} dx=0,~~\mathrm{in}~\Omega.
	\end{array} \right.
\end{equation}
Then we have the following regularity results:

(1) If $F\in L^2$, then $(u, P)\in H^2\times H^1$ and
\begin{equation}\label{stokes-2}
\| u\|_{H^2}+\bigg\|\frac{P}{\rho^\alpha}\bigg\|_{H^1}
\leq C(\bar\rho^{-\alpha}+\bar\rho^{-\alpha-\frac{q}{q-3}}\|\nabla \rho \|_{L^q}^\frac{q}{q-3})\|F\|_{L^2};
\end{equation}

(2) If $F\in L^q$ for some $q\in (3,6)$ then $(u, P)\in W^{2,q}\times W^{1,q}$ and
\begin{equation}\label{stokes-q}
\| u\|_{W^{2,q}}+\bigg\|\frac{P}{\rho^\alpha}\bigg\|_{W^{1,q}}
\leq  C(\bar\rho^{-\alpha}+\bar\rho^{-\alpha-\frac{5q-6}{2(q-3)}}\|\nabla \rho \|_{L^q}^\frac{5q-6}{2(q-3)})\|F\|_{L^q}.
\end{equation}
Here the constant C in \eqref{stokes-2} and \eqref{stokes-q} depends on $\Omega, q$.
\end{lemma}
\begin{proof}
Multiply the first equation of $\eqref{ins}_2$ by $u$ and integrate over $\Omega$, then by Cauchy's inequality,
\begin{equation}
    \int2\rho^\alpha\vert d\vert^2 dx=\int F\cdot udx\leq \|F\|_{L^2}\|u\|_{L^2}.
\end{equation}
Note that
\begin{equation}\label{deg}
    2\int \vert d\vert^2dx=\int\vert\nabla u\vert^2dx,
\end{equation}
hence, it follows from \eqref{2.3} and \eqref{basic-est} that
\begin{equation}
    \|\nabla u\|_{L^2}\leq C\bar{\rho}^{-\alpha}\| u\|_{L^2}\| F\|_{L^2}\leq C\bar{\rho}^{-\alpha}\| F\|_{L^2}.
\end{equation}

Since $\int\frac{P}{\rho^\alpha}dx =0$, according to Lemma \ref{bovosgii}, there exists a function $v\in H^1_0$, such that
\begin{equation}
    \mathrm{div}v=\frac{P}{\rho^\alpha},
\end{equation}
and
\begin{equation}
    \|\nabla v\|_{L^2}\leq C\bigg\|\frac{P}{\rho^\alpha}\bigg\|_{L^2}.
\end{equation}
Multiplying the first equation of \eqref{stokes} by $-v$, and integrating over $\Omega$, then making use of Poincar\'e's inequality, one obtains
\begin{equation}
\begin{split}
    \int\frac{P^2}{\rho^\alpha}dx
    &=-\int F\cdot vdx+2\int \mu\rho^\alpha d:\nabla v dx\\
    &\leq C\|F\|_{L^2}\|v\|_{L^2}
    +C\bar{\rho}^\alpha\|\nabla u\|_{L^2}\|\nabla v\|_{L^2}\\
    &\leq C\|F\|_{L^2}\|\nabla v\|_{L^2}\\
    &\leq C\|F\|_{L^2}\bigg\|\frac{P}{\rho^\alpha}\bigg\|_{L^2}.
\end{split}
\end{equation}
On the other hand,
\begin{equation}
    \int\frac{P^2}{\rho^\alpha}dx\geq C\bar{\rho} \int\frac{P^2}{\rho^{2\alpha}}dx,
\end{equation}
hence
\begin{equation}
    \bigg\|\frac{P}{\rho^\alpha}\bigg\|_{L^2}\leq C\bar{\rho}^{-\alpha}\|F\|_{L^2}.
\end{equation}

Rewrite \eqref{stokes}$_2$ as 
\begin{equation}
	-\mu\Delta u+\nabla\biggl(\frac{P}{\rho^\alpha}\biggr)=\frac{F}{\rho^\alpha}+2\mu\alpha \rho^{-1}\nabla \rho \cdot d
 +\alpha \rho^{-1}\nabla \rho \frac{P}{\rho^\alpha}. 
\end{equation}
Stokes estimates and Lemma \ref{G-N} yields
\begin{equation}
	\begin{aligned}
		&\|\nabla^2 u\|_{L^2}+\bigg\|\nabla\biggl(\frac{P}{\rho^\alpha}\biggr)\bigg\|_{L^2} \\	
		\leq& C\biggl(\bar\rho^{-\alpha}\|F\|_{L^2}
        +\bar\rho^{-1} \|\nabla\rho\cdot\nabla u\|_{L^2}
        +\bar\rho^{-1} \|\nabla\rho\cdot\frac{P}{\rho^\alpha}\|_{L^2}\biggr)  \\
        \leq& C\biggl(\bar\rho^{-\alpha}\|F\|_{L^2}
        +\bar\rho^{-1} \|\nabla \rho \|_{L^q} \|\nabla u\|_{L^{\frac{2q}{q-2}}}
        +\bar\rho^{-1} \|\nabla \rho \|_{L^q} \bigg\|\frac{P}{\rho^\alpha}\bigg\|_{L^{\frac{2q}{q-2}}} \biggr) \\
		\leq& C\biggl(\bar\rho^{-\alpha}\|F\|_{L^2}
        +\bar\rho^{-1} \|\nabla \rho \|_{L^q} \|\nabla u\|_{L^2}^\frac{q-3}{q} \|\nabla^2 u\|_{L^2}^\frac{3}{q}
        +\bar\rho^{-1} \|\nabla \rho \|_{L^q} \bigg\|\frac{P}{\rho^\alpha}\bigg\|_{L^2}^\frac{q-3}{q}\bigg\|\nabla\biggl(\frac{P}{\rho^\alpha}\biggr)\bigg\|_{L^2}^\frac{3}{q} \biggr). 
	\end{aligned}
\end{equation}
By Young’s inequality,
\begin{equation}
	\begin{aligned}
		&\|\nabla^2 u\|_{L^2}+\bigg\|\nabla\biggl(\frac{P}{\rho^\alpha}\biggr)\bigg\|_{L^2} \\	
		\leq& C\bar\rho^{-\alpha}\|F\|_{L^2}
        +C\bar\rho^{-\frac{q}{q-3}} \|\nabla \rho \|_{L^q}^\frac{q}{q-3} \biggl(\|\nabla u\|_{L^2}+\bigg\|\frac{P}{\rho^\alpha}\bigg\|_{L^2} \biggr)  \\
        \leq& C(\bar\rho^{-\alpha}+\bar\rho^{-\alpha-\frac{q}{q-3}}\|\nabla \rho \|_{L^q}^\frac{q}{q-3})\|F\|_{L^2}.
	\end{aligned}
\end{equation}
Similarly,
\begin{equation}
\|\nabla^2 u\|_{L^q}+\bigg\|\nabla\biggl(\frac{P}{\rho^\alpha}\biggr)\bigg\|_{L^q}
        \leq  C(\bar\rho^{-\alpha}+\bar\rho^{-\alpha-\frac{5q-6}{2(q-3)}}\|\nabla \rho \|_{L^q}^\frac{5q-6}{2(q-3)})\|F\|_{L^q}.
\end{equation}

\end{proof}


As a consequence, we have the following high-order estimate of the velocities which will be used frequently.

\begin{lemma}
	Under the assumption \eqref{a1}, it holds that 
	\begin{equation}\label{H2}
	    \| u\|_{H^2} \leq C(\bar\rho^{\frac{1}{2}-\alpha}\|\sqrt{\rho} u_t\|_{L^2}
+\bar\rho^{2-2\alpha}\|\nabla u\|_{L^2}^3),
	\end{equation}
		and
        \begin{equation}\label{W2q}
	    \| u\|_{W^{2,q}} \leq C \bar\rho^{-\alpha}  \|\rho u_t\|_{L^q} + C \bar \rho^{(1-\alpha)\frac{5q-6}{q}} \|\nabla u\|_{L^2}^{\frac{6(q-1)}{q}}.
	\end{equation}	
\end{lemma}
\begin{proof}
Let 
\begin{equation}
    F=\rho \dot u
\end{equation}
in Lemma \ref{stokes-e}. 
\eqref{stokes-2} together with \eqref{a1} and \eqref{basic-est} gives
\begin{equation}\label{d2uu}
\begin{aligned}
\|u\|_{H^2}
\leq& C(\bar\rho^{-\alpha}+\bar\rho^{-\alpha-\frac{q}{q-3}}\|\nabla \rho \|_{L^q}^\frac{q}{q-3})\|\rho \dot  u\|_{L^2}\\
\leq& C(\bar\rho^{\frac{1}{2}-\alpha}+\bar\rho^{\frac{1}{2}-\alpha-\frac{q}{q-3}}\|\nabla \rho \|_{L^q}^\frac{q}{q-3})\|\sqrt\rho u_t\|_{L^2}\\
&+C(\bar\rho^{1-\alpha}+\bar\rho^{1-\alpha-\frac{q}{q-3}}\|\nabla \rho \|_{L^q}^\frac{q}{q-3})\|u\|_{L^6}\|\nabla u\|_{L^3}\\
\leq& C(\bar\rho^{\frac{1}{2}-\alpha}+\mathcal{E}_\rho(0)^\frac{q}{q-3}\bar\rho^{\frac{1}{2}-\alpha-\frac{q}{q-3}})\|\sqrt\rho u_t\|_{L^2}\\
&+C(\bar\rho^{1-\alpha}+\mathcal{E}_\rho(0)^\frac{q}{q-3}\bar\rho^{1-\alpha-\frac{q}{q-3}})\|\nabla u\|_{L^2}^\frac{3}{2}\|\nabla u\|_{H^1}^\frac{1}{2}\\
\leq& C(\bar\rho^{\frac{1}{2}-\alpha}+\mathcal{E}_\rho(0)^\frac{q}{q-3}\bar\rho^{\frac{1}{2}-\alpha-\frac{q}{q-3}})\|\sqrt\rho u_t\|_{L^2}\\
&+C(\bar\rho^{2-2\alpha}+\mathcal{E}_\rho(0)^\frac{2q}{q-3}\bar\rho^{2-2\alpha-\frac{2q}{q-3}})\|\nabla u\|_{L^2}^3\\
\leq& C(\bar\rho^{\frac{1}{2}-\alpha}\|\sqrt{\rho} u_t\|_{L^2}
+\bar\rho^{2-2\alpha}\|\nabla u\|_{L^2}^3),
\end{aligned}
\end{equation}
provided $\bar\rho\geq 1$. 
Similarly, Gagliardo-Nirenber inequality together with \eqref{stokes-q} yields
    \begin{equation}
		\begin{aligned} 
		\| u\|_{W^{2,q}} \le & C \bar\rho^{-\alpha}  \|\rho \dot u\|_{L^q} \\
        \le & C \bar\rho^{-\alpha} (  \|\rho u_t\|_{L^q} + \bar \rho \|u\|_{L^6} \| \nabla u\|_{L^{\frac{6q}{6-q}}})\\
        \le & C \bar\rho^{-\alpha} (  \|\rho u_t\|_{L^q} + \bar \rho \|\nabla u\|_{L^2}^{\frac{6(q-1)}{5q-6}} \|\nabla  u\|_{W^{1,q}}^{\frac{4q-6}{5q-6}})\\
        \leq & \frac12 \| \nabla u\|_{W^{1,q}} + C \bar\rho^{-\alpha}  \|\rho u_t\|_{L^q} + C \bar \rho^{(1-\alpha)\frac{5q-6}{q}} \|\nabla u\|_{L^2}^{\frac{6(q-1)}{q}} \\
        \leq & C \bar\rho^{-\alpha}  \|\rho u_t\|_{L^q} + C \bar \rho^{(1-\alpha)\frac{5q-6}{q}} \|\nabla u\|_{L^2}^{\frac{6(q-1)}{q}}.
		\end{aligned}
	\end{equation}    
\end{proof}

Now we are ready to deal with an estimate to $\mathcal{E}_u(T)$.
\begin{lemma}\label{L_2}
There exists a positive constant $\Lambda_1$ such that 
\begin{equation}
\mathcal{E}_u(T)\le 2\mathcal{E}_u(0),
\end{equation}
and
\begin{equation}\label{tdu}
\bar\rho^{\alpha}\sup_{t\in[0,T]}t\|\nabla u \|_{L^2}^2
+\int_0^{T}t\|\sqrt{\rho} u_t \|_{L^2}^2dt
\leq C\bar\rho ,
\end{equation}
provided $\bar \rho >\Lambda_1=\Lambda_1(\Omega,C_0,\mu,\alpha,\|\nabla\rho_0\|_{L^q}, \| u_0\|_{H^2})$.
\end{lemma}
\begin{proof}
	Multiplying \eqref{ins}$_2$ by $u_t$, and integrating by parts, we have that 
	\begin{equation}\label{51}
		\begin{aligned}
			&\frac{d}{dt}\int \mu\rho^\alpha |d|^2dx+  \int\rho |u_t|^2dx \\
			&= -\int  \rho u \cdot \nabla u \cdot u_t dx+ \int  \mu(\rho^\alpha)_t |d|^2 dx.
		\end{aligned}
	\end{equation}
	It follows from H\"older and Sobolev inequalities that 
	\begin{equation}\label{52}
        \begin{aligned}
		\int  \rho u \cdot \nabla u \cdot u_t dx
        &\leq C\bar\rho^\frac{1}{2}\|\sqrt{\rho} u_t\|_{L^2}\|u\|_{L^6}\|\nabla u\|_{L^3} \\
        &\leq C\bar\rho^\frac{1}{2}\|\sqrt{\rho} u_t\|_{L^2}\|\nabla u\|_{L^2}^\frac{3}{2}\|\nabla u\|_{H^1}^\frac{1}{2}\\
        &\leq C\bar\rho^\frac{1}{2}\|\sqrt{\rho} u_t\|_{L^2}\|\nabla u\|_{L^2}^\frac{3}{2}(\bar\rho^{\frac{1}{2}-\alpha}\|\sqrt{\rho} u_t\|_{L^2}
        +\bar\rho^{2-2\alpha}\|\nabla u\|_{L^2}^3)^\frac{1}{2}\\
        &\leq \frac{1}{4}\|\sqrt{\rho} u_t\|_{L^2}^2
        +C(\bar\rho^{3-2\alpha}+\bar\rho^{1-2\alpha})\|\nabla u\|_{L^2}^6.
        \end{aligned}
	\end{equation}
Using the fact that
\begin{equation}\label{rhoat}
    \partial_t(\rho^\alpha)+u\cdot\nabla\rho^\alpha=0,
\end{equation}
due to \eqref{ins}$_1$ and \eqref{ins}$_4$, which together with \eqref{d2uu} yields
\begin{equation}\label{53}
\begin{aligned}
\int\mu(\rho^\alpha)_t |d|^2 dx&\le C \bar \rho^{\alpha-1} \int  \left|\nabla \rho \cdot u\right||\nabla u|^2 dx\\
&\le C \bar \rho^{\alpha-1} \|\nabla \rho \|_{L^q}\|u\|_{L^6} \|\nabla u\|_{L^{\frac{12q}{5q-6}}}^2\\
&\le C\bar{\rho}^{\alpha-1} \|\nabla \rho \|_{L^q}\|\nabla u\|_{L^2}^{\frac{5q-6}{2q}} \|\nabla u\|_{H^1}^{\frac{q+6}{2q}}\\
&\le C\mathcal{E}_\rho(0)\bar{\rho}^{\alpha-1} \|\nabla u\|_{L^2}^{\frac{5q-6}{2q}} (\bar\rho^{\frac{1}{2}-\alpha} \|\sqrt{\rho} u_t\|_{L^2}
+\bar\rho^{2-2\alpha}\|\nabla u\|_{L^2}^3 )^{\frac{q+6}{2q}}\\
&\leq \frac{1}{4} \|\sqrt{\rho} u_t\|_{L^2}^2 + C \bar \rho^{\frac{6}{q}(1-\alpha)} \|\nabla u\|_{L^2}^{4+\frac{6}{q}}
+C\bar\rho^{-\alpha+\frac{6}{q}(1-2\alpha)}\|\nabla u\|_{L^2}^\frac{2(5q-6)}{3(q-2)}.
\end{aligned}
\end{equation}
Then \eqref{a1} together with \eqref{51}-\eqref{53} implies	
\begin{equation}\label{t1}
\begin{aligned}
\frac{d}{dt}\int \mu\rho^\alpha |d|^2dx+  \int\rho |u_t|^2dx 
&\leq C(\bar\rho^{3-2\alpha}+\bar\rho^{1-2\alpha}+\bar \rho^{\frac{6}{q}(1-\alpha)}+\bar\rho^{-\alpha+\frac{6}{q}(1-2\alpha)})\|\nabla u\|_{L^2}^4\\
&\leq C\bar\rho^{3-2\alpha}\|\nabla u\|_{L^2}^4\\
&\leq C\bar\rho^{3-2\alpha}\|\nabla u\|_{L^2}^2,
\end{aligned}
\end{equation}
since $\frac{2(5q-6)}{3(q-2)}\geq 4$ and $\bar\rho>1$.

Integrating \eqref{t1} with respect to $t$ over $(0, T]$, we get
from \eqref{deg}, \eqref{2.3} and \eqref{basic-est} that
\begin{equation}
\begin{aligned}
\bar\rho^{\alpha}\sup_{t\in[0,T] }\|\nabla u \|_{L^2}^2
	+  \int_0^{T}  \|\sqrt{\rho} u_t \|_{L^2}^2dt
 &\leq \bar\rho^{\alpha} \|\nabla u_0\|_{L^2}^2
 +C\bar{\rho}^{3-2\alpha}\int_0^T\|\nabla u\|_{L^2}^2dt\\
 &\leq M\bar\rho^{\alpha}+C_1\bar{\rho}^{4-3\alpha}\\
 &\leq 2M\bar\rho^{\alpha}=2\mathcal{E}_u(0),
\end{aligned}
\end{equation}
provided 
\begin{equation}
    M \triangleq \|\nabla u_0\|_{L^2}^2,
\end{equation}
and
\begin{equation}
    \bar{\rho}\geq \biggl(\frac{C_1}{M}\biggr)^\frac{1}{4(\alpha-1)}\triangleq \Lambda_1.
\end{equation}
Finally, multiplying \eqref{t1} by $t$ and applying Gronwall’s inequality, we obtain
\begin{equation}
\begin{aligned}
&\bar\rho^{\alpha}\sup_{t\in[0,T]}t\|\nabla u \|_{L^2}^2
+\int_0^{T}t\|\sqrt{\rho} u_t \|_{L^2}^2dt\\
&\leq \int_{0}^{T} \int \mu\rho^\alpha |d|^2dxdt 
\cdot\exp{\bigg\{C\bar\rho^{3-3\alpha}\int_0^{T}  \|\nabla u\|_{L^2}^2dt\bigg\}}\\
&\leq C\bar\rho \cdot\exp{\{C\bar\rho^{4-4\alpha}\}}\\
&\leq C\bar\rho,
\end{aligned}
\end{equation}
due to $\alpha>1$.
\end{proof}
In order to close the estimate of $\mathcal{E}_\rho$, we need the following time-weight estimates.
\begin{lemma}
    It holds that 
\begin{equation}\label{trut}
    \sup_{0\le t\le T} t \|\sqrt{\rho} u_t \|_{L^2}^2 +  \bar\rho^\alpha\int_{0}^{T} t\|\nabla u_t\|^2_{L^2} dt
    \leq C\bar\rho^\alpha
\end{equation}
and
\begin{equation}\label{ttrut}
    \sup_{0\le t\le T} t^2 \|\sqrt{\rho} u_t \|_{L^2}^2 +  \bar\rho^\alpha\int_{0}^{T} t^2\|\nabla u_t\|^2_{L^2} dt
    \leq C\bar\rho.
\end{equation}
\end{lemma}

\begin{proof}
Take the $t$-derivative of the momentum equations, multiply the resulting equation by $tu_t$, and after integrating by parts, we have that 
\begin{equation}\label{k2}
    \begin{aligned}
        &\frac{t}{2} \frac{d}{dt} \int \rho |u_t|^2 dx + 2\mu t \int\rho^\alpha |d_t|^2dx\\
        =&t\int  \dv(\rho u) |u_t|^2 dx 
        -t\int  \rho u_t \cdot \nabla u \cdot u_t dx 
        +t\int  \dv(\rho u) u \cdot \nabla u \cdot u_t dx  \\
        &+2\mu t \int (\rho^\alpha)_t d: d_tdx\\
        =&\sum_{i=1}^{4} I_i.
    \end{aligned}
\end{equation}

It follows from H\"older and Sobolev inequalities that
\begin{equation}
    \begin{aligned}
        I_1\leq& C t \int |\rho \nabla u|  |u_t|^2 dx  \\
        \le &C t\bar \rho^{\frac12}\|\sqrt{\rho} u_t \|_{L^2} \|\nabla u\|_{L^3}\| u_t \|_{L^6} \\
        \le &Ct  \bar\rho^{\frac12}\|\sqrt{\rho} u_t \|_{L^2} \|\nabla u\|_{L^2}^{\frac12} \|\nabla u\|_{L^6}^{\frac12} \| \nabla u_t \|_{L^2}\\
        \le & \frac{\mu}{8} t \bar\rho^{\alpha} \|\nabla u_t\|_{L^2}^2 
        +Ct\bar \rho^{1-\alpha}\|\sqrt{\rho} u_t \|_{L^2}^2 \|\nabla u\|_{L^2} \|\nabla u\|_{L^6}\\
        \leq & \frac{\mu}{8} t \bar\rho^{\alpha} \|\nabla u_t\|_{L^2}^2 
        +Ct\bar\rho^{-\alpha} \|\nabla u\|_{L^2} \|\sqrt{\rho} u_t \|_{L^2}^2\|\nabla u\|_{H^1}\\
        \leq & \frac{\mu}{8} t \bar\rho^{\alpha} \|\nabla u_t\|_{L^2}^2
        +Ct\bar\rho^{-\alpha} \|\nabla u\|_{L^2} \|\sqrt{\rho} u_t \|_{L^2}^2
        (\bar\rho^{\frac{1}{2}-\alpha}\|\sqrt{\rho} u_t\|_{L^2}
        +\bar\rho^{2-2\alpha}\|\nabla u\|_{L^2}^3)\\
        \leq & \frac{\mu}{8} t \bar\rho^{\alpha} \|\nabla u_t\|_{L^2}^2
        +Ct\bar\rho^{\frac{1}{2}-2\alpha} \|\nabla u\|_{L^2} \|\sqrt{\rho} u_t \|_{L^2}^3
        +Ct\bar\rho^{2-3\alpha} \|\nabla u\|_{L^2}^4 \|\sqrt{\rho} u_t \|_{L^2}^2,
    \end{aligned}
\end{equation}
also $I_2$ can be estimated as follows
\begin{equation}
    \begin{aligned}
        I_2\leq& Ct\int|\nabla\rho\cdot u||u_t|^2dx\\
        \leq &C t\|\nabla \rho\|_{L^q} \|u \|_{L^{\frac{2q}{q-2}}} \| u_t \|_{L^4}^{2}\\
        \leq &Ct \bar \rho^{-\frac14} \|\nabla \rho\|_{L^q} \|u \|_{L^{2}}^{\frac{q-3}{q}} \|\nabla u \|_{L^{2}}^{\frac{3}{q}} \|\sqrt{\rho} u_t \|_{L^2}^{\frac12} \|\nabla u_t \|_{L^2}^{\frac32}  \\
        \le & \frac{\mu}{8} t \bar\rho^{\alpha} \|\nabla u_t\|_{L^2}^2 
        +C t \bar \rho^{-3\alpha-1} \|\nabla \rho\|_{L^q}^4 \|\nabla u \|_{L^{2}}^{\frac{12}{q}} \|\sqrt{\rho} u_t \|_{L^2}^2 \\
        \leq & \frac{\mu}{8} t \bar\rho^{\alpha} \|\nabla u_t\|_{L^2}^2 
        + Ct \mathcal{E}_\rho(0)^4 \bar\rho^{-3\alpha-1} \|\nabla u \|_{L^2}^\frac{12}{q} \|\sqrt{\rho} u_t \|_{L^2}^2.
    \end{aligned}
\end{equation}
Similarly, $I_3$ can be bounded by
\begin{equation}
    \begin{aligned}
        I_3\leq & t\int |\nabla \rho \cdot u| |u| | \nabla u| |u_t| dx\\
        \leq & Ct \|\nabla \rho\|_{L^q}  \|u\|_{L^6}^2  \|\nabla u\|_{L^{\frac{2q}{q-2}}} \|u_t\|_{L^6} \\
        \leq & \frac{\mu}{8} t \bar\rho^{\alpha} \|\nabla u_t\|_{L^2}^2 + Ct \bar\rho^{ - \alpha} \|\nabla \rho \|_{L^q}^2 \|\nabla u\|_{L^2}^\frac{6(q-1)}{q}\|\nabla u\|_{H^1}^\frac{6}{q}\\
        \leq & \frac{\mu}{8} t \bar\rho^{\alpha} \|\nabla u_t\|_{L^2}^2 
        +C\mathcal{E}_\rho(0)^2t \bar\rho^{ - \alpha} \|\nabla u\|_{L^2}^\frac{6(q-1)}{q}(\bar\rho^{\frac{1}{2}-\alpha}\|\sqrt{\rho} u_t\|_{L^2}
        +\bar\rho^{2-2\alpha}\|\nabla u\|_{L^2}^3)^\frac{6}{q}\\
        \leq & \frac{\mu}{8} t \bar\rho^{\alpha} \|\nabla u_t\|_{L^2}^2
        +Ct\rho^{ -\alpha+2(1-\alpha)\frac{6}{q}}
        \|\nabla u\|_{L^2}^{6+\frac{12}{q}}\\
        &+Ct\bar\rho^{ -\alpha+\frac{6}{q}(\frac{1}{2}-\alpha)}(\|\nabla u\|_{L^2}^2
        \|\sqrt{\rho} u_t\|_{L^2}^2+\|\nabla u\|_{L^2}^{6+\frac{6}{q-3}}).
    \end{aligned}
\end{equation}
Using \eqref{rhoat}, we have
\begin{equation}\label{I4}
    \begin{aligned}
        I_4\leq & Ct\bar\rho^{\alpha-1}\int |\nabla \rho \cdot u| |\nabla  u||\nabla  u_t| dx\\
        \leq & Ct\bar \rho^{\alpha-1}  \|\nabla \rho\|_{L^q}  \|u\|_{L^\frac{3q}{q-3}}  \|\nabla u\|_{L^6} \|\nabla u_t\|_{L^2}\\
        \leq & \frac{\mu}{8} t\bar \rho^{\alpha} \|\nabla u_t\|_{L^2}^2
        +Ct\bar \rho^{\alpha-2} \|\nabla \rho \|_{L^q}^2  \|\nabla u\|_{L^2}^{3-\frac{6}{q}} \|\nabla u\|_{L^6}^{1+\frac{6}{q}}\\
        \leq & \frac{\mu}{8}t \bar\rho^{\alpha} \|\nabla u_t\|_{L^2}^2 
        + C\mathcal{E}_\rho(0)^2 t\bar \rho^{\alpha-2}  \|\nabla u\|_{L^2}^{3-\frac{6}{q}} \|\nabla u\|_{H^1}^{1+\frac{6}{q}}\\
        \leq & \frac{\mu}{8}t \bar \rho^{\alpha} \|\nabla u_t\|_{L^2}^2 
        +Ct \bar\rho^{\alpha-2}  \|\nabla u\|_{L^2}^{3-\frac{6}{q}}(\bar\rho^{\frac{1}{2}-\alpha}\|\sqrt{\rho} u_t\|_{L^2}
        +\bar\rho^{2-2\alpha}\|\nabla u\|_{L^2}^3)^{1+\frac{6}{q}}\\
        \leq & \frac{\mu}{8}t \bar \rho^{\alpha} \|\nabla u_t\|_{L^2}^2  
        +Ct \bar\rho^{-\frac{3}{2}+(\frac{1}{2}-\alpha)\frac{6}{q}}\|\nabla u\|_{L^2}^{3-\frac{6}{q}}\|\sqrt{\rho} u_t\|_{L^2}^{1+\frac{6}{q}}
        +Ct \bar\rho^{-\alpha+2(1-\alpha)\frac{6}{q}}\|\nabla u\|_{L^2}^{6+\frac{12}{q}}.
    \end{aligned}
\end{equation}
Combining all the above estimates \eqref{k2}–\eqref{I4},
\eqref{a1} and \eqref{basic-est}, we deduce
\begin{equation}\label{k3}
    \begin{aligned}
        &\frac{d}{dt}t \int \rho |u_t|^2 dx 
        + t \bar\rho^\alpha \int|\nabla u_t|^2dx\\
        \leq & \int \rho |u_t|^2 dx +Ct\bar\rho^{\frac{1}{2}-2\alpha} \|\nabla u\|_{L^2} \|\sqrt{\rho} u_t \|_{L^2}^3
        +Ct\bar\rho^{2-3\alpha} \|\nabla u\|_{L^2}^4 \|\sqrt{\rho} u_t \|_{L^2}^2\\
        &+ Ct \mathcal{E}_\rho(0)^4 \bar\rho^{-3\alpha-1} \|\nabla u \|_{L^2}^\frac{12}{q} \|\sqrt{\rho} u_t \|_{L^2}^2
        +Ct\bar\rho^{ -\alpha+\frac{6}{q}(\frac{1}{2}-\alpha)}(\|\nabla u\|_{L^2}^2
        \|\sqrt{\rho} u_t\|_{L^2}^2+\|\nabla u\|_{L^2}^{6+\frac{6}{q-3}})\\
        &+Ct\rho^{ -\alpha+2(1-\alpha)\frac{6}{q}}
        \|\nabla u\|_{L^2}^{6+\frac{12}{q}}
        +Ct \bar\rho^{-\frac{3}{2}+(\frac{1}{2}-\alpha)\frac{6}{q}}\|\nabla u\|_{L^2}^{3-\frac{6}{q}}\|\sqrt{\rho} u_t\|_{L^2}^{1+\frac{6}{q}}\\
        \leq &\int \rho |u_t|^2 dx
        +Ct\bar\rho^{\frac{1}{2}-2\alpha} \|\nabla u\|_{L^2} \|\sqrt{\rho} u_t \|_{L^2}^3
        +Ct \bar\rho^{-\frac{3}{2}+(\frac{1}{2}-\alpha)\frac{6}{q}}\|\nabla u\|_{L^2}^{3-\frac{6}{q}}\|\sqrt{\rho} u_t\|_{L^2}^{1+\frac{6}{q}}\\
        &+Ct(\bar\rho^{2-3\alpha}+\bar\rho^{ -\frac{1}{2}+\frac{6}{q}(\frac{1}{2}-\alpha)}) \|\nabla u\|_{L^2}^2 \|\sqrt{\rho} u_t \|_{L^2}^2\\
        &+Ct(\rho^{ -\alpha+2(1-\alpha)\frac{6}{q}}+\bar\rho^{ -\frac{1}{2}+\frac{6}{q}(\frac{1}{2}-\alpha)})
        \|\nabla u\|_{L^2}^8,
    \end{aligned}
\end{equation}
due to $q\in (3,6)$ and \eqref{a1}. Thus, Gronwall's inequality yields
\begin{equation}
    \begin{aligned}
        &\sup_{0\le t\le T} t \|\sqrt{\rho} u_t \|_{L^2}^2 
        +\bar\rho^\alpha\int_{0}^{T} t\|\nabla u_t\|^2_{L^2} dt\\
        \leq & C\int_0^T \left( t(\bar\rho^{ -\alpha+2(1-\alpha)\frac{6}{q}}+\bar\rho^{ -\frac{1}{2}+\frac{6}{q}(\frac{1}{2}-\alpha)})
        \|\nabla u\|_{L^2}^8
        +\|\sqrt{\rho} u_t\|_{L^2}^2 \right) dt\\
        &\cdot\exp{\left\{\int_0^T \left(\bar\rho^{\frac{1}{2}-2\alpha} \|\nabla u\|_{L^2} \|\sqrt{\rho} u_t \|_{L^2}
        +\bar\rho^{-\frac{3}{2}+(\frac{1}{2}-\alpha)\frac{6}{q}}\|\nabla u\|_{L^2}^{3-\frac{6}{q}}\|\sqrt{\rho} u_t\|_{L^2}^{\frac{6}{q}-1}\right)dt\right\}}\\
        &\cdot\exp{\left\{ \bigl(\bar\rho^{2-3\alpha}+\bar\rho^{ -\frac{1}{2}+\frac{6}{q}(\frac{1}{2}-\alpha)}\bigr)\int_0^T \|\nabla u\|_{L^2}^2dt\right\}}.
    \end{aligned}
\end{equation}
Taking advantage of \eqref{a1}, \eqref{basic-est} and \eqref{tdu}, we obtain
\begin{equation}\label{111}
    \begin{aligned}
        &\int_0^T t(\bar\rho^{ -\alpha+2(1-\alpha)\frac{6}{q}}+\bar\rho^{ -\frac{1}{2}+\frac{6}{q}(\frac{1}{2}-\alpha)})
        \|\nabla u\|_{L^2}^8 dt\\
        \leq& C (\bar\rho^{ -\alpha+2(1-\alpha)\frac{6}{q}}+\bar\rho^{ -\frac{1}{2}+\frac{6}{q}(\frac{1}{2}-\alpha)})
        \sup_{t\in[0,T]}\|\nabla u \|_{L^2}^4
        \cdot\sup_{t\in[0,T]}t\|\nabla u \|_{L^2}^2
        \cdot\int_0^T\|\nabla u \|_{L^2}^2dt\\
        \leq& C(\bar\rho^{ -\alpha+2(1-\alpha)\frac{6+q}{q}}+\bar\rho^{ -\frac{1}{2}+\frac{6}{q}(\frac{1}{2}-\alpha)+2(1-\alpha)}).
    \end{aligned}
\end{equation}
H\"older's inequality yields
\begin{equation}
    \begin{aligned}
        &\int_0^T \bar\rho^{\frac{1}{2}-2\alpha} \|\nabla u\|_{L^2} \|\sqrt{\rho} u_t \|_{L^2} dt\\
        \leq& \bar\rho^{\frac{1}{2}-2\alpha}
        \left(\int_0^T\|\nabla u\|_{L^2}^2dt\right)^\frac{1}{2}
        \left(\int_0^T\|\sqrt{\rho} u_t\|_{L^2}^2dt\right)^\frac{1}{2}\\
        \leq& C\bar\rho^{1-2\alpha},
    \end{aligned}
\end{equation}
and
\begin{equation}\label{113}
    \begin{aligned}
        &\int_0^T \bar\rho^{-\frac{3}{2}+(\frac{1}{2}-\alpha)\frac{6}{q}}\|\nabla u\|_{L^2}^{3-\frac{6}{q}}\|\sqrt{\rho} u_t\|_{L^2}^{\frac{6}{q}-1} dt\\
        \leq& \bar\rho^{-\frac{3}{2}+(\frac{1}{2}-\alpha)\frac{6}{q}}
        \left(\int_0^T\|\nabla u\|_{L^2}^2dt\right)^\frac{3q-6}{2q}
        \left(\int_0^T\|\sqrt{\rho} u_t\|_{L^2}^2dt\right)^\frac{6-q}{2q}\\
        \leq& C\bar\rho^{-2\alpha}.
    \end{aligned}
\end{equation}
Hence, collecting all the estimates \eqref{111}-\eqref{113}, one gets
\begin{equation}\label{trut1}
    \sup_{0\le t\le T} t \|\sqrt{\rho} u_t \|_{L^2}^2 +  \bar\rho^\alpha\int_{0}^{T} t\|\nabla u_t\|^2_{L^2} dt
    \leq C(\bar\rho^{-A}+\bar\rho^\alpha)\cdot
    \exp{\{C\bar\rho^{-B}\}}
    \leq C\bar\rho^\alpha.
\end{equation}
with some $A,B>0$.

On the other hand, multiplying \eqref{k3} by $t$, one has
\begin{equation}
    \begin{aligned}
        &\frac{d}{dt}\frac{t^2}{2} \int \rho |u_t|^2 dx 
        +t^2\bar\rho^\alpha \int|\nabla u_t|^2dx\\
        \leq & t\int \rho |u_t|^2 dx
        +Ct^2\bar\rho^{\frac{1}{2}-2\alpha} \|\nabla u\|_{L^2} \|\sqrt{\rho} u_t \|_{L^2}^3
        +Ct^2 \bar\rho^{-\frac{3}{2}+(\frac{1}{2}-\alpha)\frac{6}{q}}\|\nabla u\|_{L^2}^{3-\frac{6}{q}}\|\sqrt{\rho} u_t\|_{L^2}^{1+\frac{6}{q}}\\
        &+Ct^2(\bar\rho^{2-3\alpha}+\bar\rho^{ -\frac{1}{2}+\frac{6}{q}(\frac{1}{2}-\alpha)}) \|\nabla u\|_{L^2}^2 \|\sqrt{\rho} u_t \|_{L^2}^2\\
        &+Ct^2(\rho^{ -\alpha+2(1-\alpha)\frac{6}{q}}+\bar\rho^{ -\frac{1}{2}+\frac{6}{q}(\frac{1}{2}-\alpha)})
        \|\nabla u\|_{L^2}^8.
    \end{aligned}
\end{equation}
Applying Gronwall’s inequality to arrive at
\begin{equation}
    \begin{aligned}
        &\sup_{0\le t\le T} t^2 \|\sqrt{\rho} u_t \|_{L^2}^2 +  \bar\rho^\alpha\int_{0}^{T} t^2\|\nabla u_t\|^2_{L^2} dt\\
        \leq & C\int_0^T \left( t^2(\bar\rho^{ -\alpha+2(1-\alpha)\frac{6}{q}}+\bar\rho^{ -\frac{1}{2}+\frac{6}{q}(\frac{1}{2}-\alpha)})
        \|\nabla u\|_{L^2}^8
        +t\|\sqrt{\rho} u_t\|_{L^2}^2 \right) dt
        \cdot\exp{\{\bar\rho^{-B}\}},
    \end{aligned}
\end{equation}
and \eqref{tdu} together with \eqref{trut} yields
\begin{equation}
    \begin{aligned}
        &\int_0^T t^2(\bar\rho^{ -\alpha+2(1-\alpha)\frac{6}{q}}+\bar\rho^{ -\frac{1}{2}+\frac{6}{q}(\frac{1}{2}-\alpha)})
        \|\nabla u\|_{L^2}^8 dt\\
        \leq& C (\bar\rho^{ -\alpha+2(1-\alpha)\frac{6}{q}}+\bar\rho^{ -\frac{1}{2}+\frac{6}{q}(\frac{1}{2}-\alpha)})
        \sup_{t\in[0,T]}\|\nabla u \|_{L^2}^2
        \cdot\sup_{t\in[0,T]}t^2\|\nabla u \|_{L^2}^4
        \cdot\int_0^T \|\nabla u \|_{L^2}^2dt\\
        \leq& C(\bar\rho^{ -1+2(1-\alpha)\frac{6}{q} +3(1-\alpha)}+\bar\rho^{ -\frac{1}{2}+\frac{6}{q}(\frac{1}{2}-\alpha)+3(1-\alpha)}).
    \end{aligned}
\end{equation}
Hence it follows immediately by \eqref{trut1} that,
\begin{equation}
    \sup_{0\le t\le T} t^2 \|\sqrt{\rho} u_t \|_{L^2}^2 +  \bar\rho^\alpha\int_{0}^{T} t^2\|\nabla u_t\|^2_{L^2} dt
    \leq C(\bar\rho^{-A_1}+\bar\rho)\cdot
    \exp{\{\bar\rho^{-B_1}\}}
    \leq C\bar\rho,
\end{equation}
with some $A_1,B_1>0$.
\end{proof}
Finally, we are about to finish the bound of  $\mathcal{E}_\rho$, the key observation is that $\|\nabla u\|_{L^1_tL^\infty_x}$ is uniformly bounded with respect to time $T$.
\begin{lemma}\label{L_1}
	There exists a positive constant $\Lambda_2$ such that 
	\begin{equation}
		\mathcal{E}_\rho(T)\le  2 \mathcal{E}_\rho(0),
	\end{equation}
	provided $\bar\rho>\Lambda_2=\Lambda_2(\Omega,C_0,\mu,\alpha,\|\nabla\rho_0\|_{L^q}, \| u_0\|_{H^2})$.
\end{lemma}
\begin{proof}
	It follows from \eqref{ins}$_1$  that  
	\begin{equation}
		\nabla \rho_t + u \cdot \nabla^2 \rho + \nabla u \cdot \nabla \rho =0.
	\end{equation}
	Multiplying the above equation by $|\nabla \rho|^{q-2}\nabla \rho$ and then integrating by parts, we have 
	\begin{equation}\label{dr}
		\begin{aligned}
			\frac{1}{p}\frac{d}{dt} \|\nabla \rho\|_{L^q}^q  = & - \int \nabla \rho \cdot \nabla u \cdot \nabla \rho |\nabla \rho|^{q-2} dx  \\
			\le & C \|\nabla u\|_{L^\infty} \|\nabla \rho \|_{L^q}^q .
		\end{aligned}
	\end{equation}
It follows from \eqref{W2q} and Gagliardo-Nirenber inequality that

	\begin{equation}\label{kkk}
		\begin{aligned}
			\int_0^T \|\nabla u\|_{L^\infty} dt
   &\le \int_0^T \|\nabla u\|_{W^{1,q}} dt\\
   &\le C \bar\rho^{-\alpha} \int_0^T  \|\rho u_t\|_{L^q}   dt + \bar \rho^{(1-\alpha)\frac{5q-6}{q}} \int_0^T  \|\nabla u\|_{L^2}^{\frac{6(q-1)}{q}} dt.
		\end{aligned}
	\end{equation}
For the first term on right-hand side of the above inequality, after using Gagliardo-Nirenber inequality and \eqref{trut} and \eqref{ttrut}, we have 
    \begin{equation}
        \begin{aligned}
            &\int_{0}^{T}  \|\rho u_t\|_{L^q} dt \\
            \le & \int_{0}^{T}  \bar\rho^{\frac{5q-6}{4q}}\|\sqrt{\rho}u_t\|_{L^2}^{\frac{6-q}{2q}} \|\nabla u_t\|_{L^2}^{\frac{3(q-2)}{2q}} dt\\
            \le & C \bar\rho^{\frac{5q-6}{4q}} \left(\sup_{0\le t\le \min\{1,T\}} t \|\sqrt{\rho}u_t\|_{L^2}^2 dt\right)^{\frac{6-q}{4q}}\\
            &\cdot \left( \int_{0}^{\min\{1,T\}} t \|\nabla u_t\|_{L^2}^2 dt \right)^{\frac{3(q-2)}{4q}}\left( \int_{0}^{\min\{1,T\}} t^{-\frac{2q}{q+6}}  dt \right)^{\frac{q+6}{4q}} \\
            &+ C \bar\rho^{\frac{5q-6}{4q}} \left(\sup_{\min\{1,T\}\le t\le T} t^2 \|\sqrt{\rho}u_t\|_{L^2}^2 dt\right)^{\frac{6-q}{4q}} \\
            &\cdot\left( \int_{\min\{1,T\}}^{T} t^2 \|\nabla u_t\|_{L^2}^2 dt \right)^{\frac{3(q-2)}{4q}} \left( \int_{\min\{1,T\}}^{T} t^{-\frac{4q}{q+6}}  dt \right)^{\frac{q+6}{4q}} \\
            \le & C \bar\rho^{\frac{5q-6}{4q} + \alpha \frac{6-q}{4q} } + C \bar\rho^{\frac{5q-6}{4q} + \frac{6-q}{4q} + (1-\alpha) \frac{3(q-2)}{4q}} \\
            \le & C \bar\rho^{\frac{5q-6}{4q} + \alpha \frac{6-q}{4q} },
        \end{aligned}
    \end{equation}
which together with \eqref{kkk}, yields that after using \eqref{basic-est},
    \begin{equation}
		\begin{aligned}
			\int_0^T \|\nabla u\|_{L^\infty} dt \le & C \bar\rho^{-\alpha} \int_0^T  \|\rho u_t\|_{L^q}   dt + \bar \rho^{(1-\alpha)\frac{5q-6}{q}} \int_0^T  \|\nabla u\|_{L^2}^{\frac{6(q-1)}{q}} dt\\
            \le & C \bar\rho^{\frac{5q-6}{4q} + \alpha \frac{6-q}{4q} -\alpha } +  C \bar \rho^{(1-\alpha)\frac{5q-6}{q} + (1-\alpha)}\\
            \le & C \bar\rho^{\frac{5q-6}{4q}(1 -\alpha) } +  C \bar \rho^{(1-\alpha)\frac{5q-6}{q} + (1-\alpha)}\\
            \leq & C \bar\rho^{-D},
		\end{aligned}
	\end{equation}
 where
 \begin{equation*}
     D=\max{\left\{\frac{5q-6}{4q}(\alpha-1), \frac{6(q-1)}{q}(\alpha-1)\right\}}.
 \end{equation*}
    Finally, note
    \begin{equation}
        \mathcal{E}_\rho(0)= \|\nabla \rho_0\|_{L^q},
    \end{equation}
    then Gronwall’s inequality together with \eqref{dr} yields
    \begin{equation}
        \sup_{0\le t\le T} \|\nabla \rho\|_{L^q}  
        \le \exp{\left\{C_2\bar\rho^{-D}\right\}}\|\nabla \rho_0\|_{L^q}\leq 2\mathcal{E}_\rho(0),
    \end{equation}
    provided 
    \begin{equation}
        \bar{\rho}\geq \biggl(\frac{C_2}{\log 2}\biggr)^\frac{1}{D}\triangleq \Lambda_2.
    \end{equation}
\end{proof}
\textbf{Proof of Prosition \ref{pr}}
    Proposition \ref{pr} is a direct consequence of Lemma \ref{L_2} and \ref{L_1}.


\section{Proof of Theorem \ref{global}}
According to Theorem \ref{local}, there exists a $\Tilde{T}>0$ such that the density-dependent Navier–Stokes system \eqref{ins}-\eqref{bc} has a unique local strong solution $(\rho, u, P)$ on $[0, \Tilde{T}]$. We use the a priori estimates, Proposition \ref{pr} to extend the local strong solution to all time.

Due to 
\begin{equation}
    \|\nabla \rho_0\|_{L^q}=\mathcal{E}_\rho(0)<3\mathcal{E}_\rho(0), \ 
    \|\nabla u_0\|_{L^2}^2=M<3M,
\end{equation}
and the local regularity results \eqref{l-r}, there exists a $T_1\in(0, \Tilde{T})$ such that
\begin{equation}
    \sup_{0\le t\le T_1} \|\nabla \rho\|_{L^q}\leq 3\mathcal{E}_\rho(0),\ 
    \sup_{0\le t\le T_1} \|\nabla u_0\|_{L^2}^2 \leq 3M.
\end{equation}
Set
\begin{equation}
    T^\ast=\sup\{T| (\rho, u, P)\ \mathrm{is}\ \mathrm{a}\ \mathrm{strong}\ \mathrm{solution}\ \mathrm{to}\ \eqref{ins}-\eqref{bc}\ \mathrm{on}\ [0,T]\},
\end{equation}
\begin{equation}
T_1^\ast=\sup\left\{T\bigg| 
\begin{array}{l}
(\rho, u, P)\ \mathrm{is}\ \mathrm{a}\ \mathrm{strong}\ \mathrm{solution}\ \mathrm{to}\ \eqref{ins}-\eqref{bc}\ \mathrm{on}\ [0,T],\\
\sup_{0\le t\le T_1} \|\nabla \rho\|_{L^q}\leq 3\mathcal{E}_\rho(0),\ 
\sup_{0\le t\le T_1} \|\nabla u_0\|_{L^2}^2 \leq 3M.
\end{array} \right\}.
\end{equation}
Then $T^\ast_1\geq T_1>0$. Recalling Proposition \ref{pr}, it’s easy to verify
\begin{equation}
    T^\ast=T^\ast_1.
\end{equation}
provided that $\bar\rho>\Lambda_0$ as assumed. 

We claim that $T^\ast=\infty$. Otherwise, assume that
$T^\ast<\infty$. By virtue of Proposition \ref{pr}, for every $t\in[0, T^\ast)$, it holds that
\begin{equation}
    \sup_{0\le t\le T_1} \|\nabla \rho\|_{L^q}\leq 2\mathcal{E}_\rho(0),\ 
    \sup_{0\le t\le T_1} \|\nabla u_0\|_{L^2}^2 \leq 2M,
\end{equation}
which contradicts the blowup criterion \eqref{blow-up}. Hence we finish the proof of Theorem \ref{global}.

\section*{Conflict-of-interest statement}
All authors declare that they have no conflicts of interest.

\section*{Acknowledgments}
X.-D. Huang is partially supported by NNSFC Grant Nos. 11971464, 11688101 and CAS
Project for Young Scientists in Basic Research, Grant No. YSBR-031, National Key R$\&$D Program of China, Grant No. 2021YFA1000800.
J.-X. Li was supported in part
by Zheng Ge Ru Foundation, Hong Kong RGC Earmarked Research Grants CUHK-14301421, CUHK-14300819,
CUHK-14302819, CUHK-14300917, the key project of NSFC (Grant No. 12131010)  the Shun Hing Education and Charity Fund. Part of this work was done when R. Zhang was visiting the Institute of Mathematical Sciences at the Chinese University of Hong Kong. They would like to thank the institute for its hospitality.

\end{document}